\documentclass[10pt]{amsart}

\usepackage{amsmath}
\usepackage{amsfonts}
\usepackage{amssymb}
\usepackage{mathrsfs}
\usepackage{textcomp}

\newtheorem{theorem}{Theorem}[section]
\newtheorem{lemma}{Lemma}[section]
\newtheorem{prop}{Proposition}[section]
\newtheorem{claim}{Claim}[section]
\newtheorem{definition}{Definition}[section]
\newtheorem{corollary}{Corollary}[section]
\newtheorem{remark}{Remark}[section]
\newtheorem{problem}{Problem}[section]

\newcommand{\ZZ}{\mathbb{Z}}
\newcommand{\QQ}{\mathbb{Q}}
\newcommand{\RR}{\mathbb{R}}
\newcommand{\CC}{\mathbb{C}}

\newcommand{\NS}[1]{\mathrm{NS}(#1)}
\newcommand{\NSQ}[1]{\mathrm{NS}_{\QQ}(#1)}
\newcommand{\NSR}[1]{\mathrm{NS}_{\RR}(#1)}
\newcommand{\MBM}[1]{\mathrm{MBM}(#1)}
\newcommand{\NEF}[1]{\mathrm{Nef}(#1)}
\newcommand{\POSNS}[1]{\mathscr{C}_{\mathrm{NS}}{(#1)}}
\newcommand{\BPOSNS}[1]{\partial \overline{\mathscr{C}_{\mathrm{NS}}(#1)}}
\newcommand{\MON}[1]{\operatorname{Mon}(#1)}
\newcommand{\MONHDG}[1]{\operatorname{Mon}^{\mathrm{Hdg}}(#1)}
\newcommand{\KEH}[1]{\mathscr{K}(#1)}
\newcommand{\POS}[1]{\mathscr{C}(#1)}
\newcommand{\LL}{\Lambda}
\newcommand{\BLL}{\overline{\Lambda}}

\newcommand{\LLC}{\Lambda_{\mathbb{C}}}
\newcommand{\LR}{\Lambda_{\mathbb{R}}}
\newcommand{\BLLR}{\overline{\Lambda}_{\mathbb{R}}}
\newcommand{\WR}{W_{\mathbb{R}}}
\newcommand{\RC}[1]{\mathscr{N}(#1)}
\newcommand{\PE}[1]{\mathrm{Eff}(#1)}
\newcommand{\AMP}[1]{\mathrm{Amp}(#1)}
\newcommand{\AUT}[1]{\operatorname{Aut}(#1)}
\newcommand{\GR}[1]{\operatorname{Gr}_{++}(#1)}
\newcommand{\RK}[1]{\operatorname{rank}(#1)}
\newcommand{\XF}{\mathscr{X}}
\newcommand{\LF}{\mathscr{L}}
\newcommand{\DEF}[1]{\mathrm{Def}(#1)}

\newcommand{\MF}[1]{\mathfrak{M}_{#1}}
\newcommand{\PF}{\mathscr{P}}
\newcommand{\PP}[1]{\mathbb{P}(#1)}
\newcommand{\PD}[1]{\Omega_{#1}}
\newcommand{\BG}{\bar{\Gamma}}

\title{On subgroups of an automorphism group of an irreducible symplectic manifold}
\author{Daisuke MATSUSHITA}

\address{Division of Mathematics, Graduate School of Science,
         Hokkaido University,  Sapporo, 060-0810 Japan}
\thanks{* Partially supported by Grand-in-Aid \# 18684001
 (Japan Society for Promortion of Sciences).} 
\email{matusita@math.sci.hokudai.ac.jp}

\begin{document}
\maketitle

\begin{abstract}
 Let $X$ be an irreducible symplectic manifold and $L$ a nef line bundle on $X$ which
 is isotropic with respect to the Beauville-Bogomolov quadratic form. It is known that
 a subgroup $\AUT{X,L}$ of an automorphism group of $X$ which fix $L$ is almost abelian. 
 We give a formula of the rank of $\AUT{X,L}$ in terms of MBM divisors. We also
 prove that the nef cone of $X$ cut out MBM classes, which is a generalization
 of Kovac's structure theorem of nef cones of $K3$ surfaces.
\end{abstract}

%%%_.
\section{Introduction}

We start with recalling definitions of a Neron-Severi group, an ample cone and  a nef cone.
\begin{definition}
 Let $X$ be a compact K\"ahler manifold. A Neron-Severi group $\NS{X}$ is 
 a subgroup of $H^2 (X,\ZZ)$ defined by
$$
 \NS{X} := H^{1,1}(X,\RR) \cap H^2 (X,\ZZ).
$$
 We denote by $\NSR{X}$ the $\RR$-vector space generated by $\NS{X}$.
 An ample cone $\AMP{X}$ of $X$ is the cone in $\NSR{X}$ defined by
$$
 \AMP{X} := \NSR{X} \cap \KEH{X},
$$
 where $\KEH{X}$ is the K\"ahler cone of $X$. 
 The closure of $\AMP{X}$ in $\NSR{X}$ is said to be nef cone and
 denoted by $\NEF{X}$.
\end{definition}
 We also recall the definition of 
 irreducible symplectic manifolds.
\begin{definition}
 A compact K\"ahler manifold $X$ is said to be irreducible symplectic
 if $X$ has the following three properties:
\begin{itemize}
 \item[(1)] $X$ is simply connected;
 \item[(2)] $X$ carries a holomorphic symplectic form and;
 \item[(3)] $\dim H^0 (X,\Omega_X^2) = 1$.
\end{itemize}
\end{definition}
 A $K3$ surface has above three properties and gives a plain example of 
 irreducible symplectic manifolds. It is expected that $K3$ surfaces
 and irreducible symplectic manifolds share many geometric properties.
 One of the biggest geometric features of $K3$ surfaces
 is Global Torelli theorem, 
 which was obtained in \cite{MR0447635}, \cite{MR596874} and \cite{MR0284440}.
\begin{theorem}[Global Torelli Theorem for projective $K3$ surfaces]
 Let $X$ and $X'$ be projective $K3$ surfaces. Assume that
 there exists an isometry  $\phi : H^2 (X',\ZZ) \to H^2 (X,\ZZ)$ with
 respect to the cup products. If $\phi $ 
 respects the Hodge structure and 
  $\phi (\AMP{X'})\cap \AMP{X} \ne \emptyset$,
 there exists an automorphism $\Phi$
 such that the induced morphism $\Phi^*$ on cohomologies
 coincides with $\phi$.
\end{theorem}
 A higher dimensional analogue of Global Torelli Theorem was obtained
 by Verbitsky in \cite[Theorem 1.18]{MR3161308}. 
 After introducing the definition of monodromy groups, we state 
 it in a form suitable for use in this paper 
 according to \cite[Theorem 1.3 (2)]{Markman_2011}.
\begin{definition}\label{053649_30Mar18}
Let $X$ be an irreducible symplectic manifold.
We denote by $q_X$ the Beauville-Bogomolov quadratic form on
$H^2 (X,\ZZ)$.
Let $O(H^2(X,\ZZ),q_X)$ be an isometry group with respect to $q_X$.
Let us consider smooth morphisms
$\mathscr{X} \to (B,o)$ such that $(B,o)$ is an analytic space 
with the reference point $o$ and the fibre at $o$ is isomorphic to $X$.
We note that $B$ may have any kind of singularities.
For such a smooth morphism, we have a natural 
representation $\pi_1 (B,o) \to O(H^2 (X,\ZZ))$.
A subgroup $\MON{X}$ of $O(H^2(X,\ZZ), q_X)$
is the subgroup generated by
the images of all such representations.
\end{definition}
\begin{theorem}[Global Torelli Theorem for projective 
  irreducible symplectic manifolds]\label{Global_Torelli_Theorem}
  Let $X$ be a projective irreducible symplectic manifold.
  Assume that there exists an element $\phi$ of $\MON{X}$ which
  respects the Hodge structures and
  $\phi (\AMP{X}) \cap \AMP{X} \ne \emptyset$.
  Then there exists an automorphism $\Phi$
  of $X$ such that the induced automorphism 
  $\Phi^*$ on $H^2 (X,\ZZ)$ coincides with $\phi$.
\end{theorem}
By the above theorem, we will have automorphisms of 
projective irreducible symplectic
manifolds if 
we construct elements of $\MON{X}$ which satisfies the assumptions
of Theorem \ref{Global_Torelli_Theorem}.
In this note, we will construct such elements of $\MON{X}$ and 
give three applications.
The first application concerns with 
the structure of  nef cones
of projective irreducible symplectic manifolds. By \cite{MR1992275} and \cite{Boucksom2001},
we have the following structure theorem of a K\"{a}hler cone of 
an irreducible symplectic manifold.
\begin{theorem}
 Let $X$ be an irreducible symplectic manifold and
 $\RC{X}$ the set of rational curves on $X$.
 We define the positive cone 
 $\POS{X}$ in $H^{1,1}(X,\RR)$ by
$$
 \POS{X} := \{x \in H^{1,1}(X,\RR)| q_X (x) > 0, q_X (x,\kappa )> 0\}
$$
 where $q_X$ is the Beauville-Bogomolov form and $\kappa$ is a K\"ahler class.
 Then
$$
 \KEH{X} = \{ x\in \POS{X} |\forall e\in \RC{X}, x.e > 0 \}
$$
\end{theorem}
 We denote by $\POSNS{X}$ the intersection of $\NSR{X}$ and $\POS{X}$.
 A nef cone $\NEF{X}$ of an irreducible symplectic manifold $X$
 can be described as follows:
$$
 \NEF{X} = \{
     x \in \overline{\POSNS{X}} | \forall e \in \RC{X}, x.e \geq 0
           \}
$$
 where  $\overline{\phantom{X}}$ stands for the closure in $\NSR{X}$.
 On the other hand, Kovacs gave a description of an effective cone 
 of a $K3$ surface in \cite{MR1314742}, whose dual cone 
 with respect to the cup product is a nef cone. 
\begin{theorem}[ { \cite[Corollary 1]{MR1314742} } ]
 Let $X$ be a projective $K3$ surface whose Picard number
 is greater than two. We denote by $\mathscr{N}(X)$ 
 the set of $(-2)$-curves on $X$.
 If $\mathscr{N}(X) = \emptyset$, then
 an effective cone $\PE{X}$ of $X$ coincides with $\overline{\POSNS{X}}$.
 If $\mathscr{N}(X) \ne \emptyset$, then
$$
 \PE{X} =
 \overline{\sum_{e \in \mathscr{N}(X)} \RR_+ e} .
$$
\end{theorem}
 If we consider the dual statement of the above theorem, 
 we find that $\NEF{X}$ coincides with $\overline{\POSNS{X}}$ 
 if $\RC{X} = \emptyset$. If $\RC{X} \ne \emptyset$,
$$
 \NEF{X} =  \{x \in \NSR{X} | \forall e \in \mathscr{N}(X), x.e \ge 0 \} 
$$ 
 It is a natural question 
 whether a nef cone of an irreducible symplectic manifold
 has a similar structure. We give a positive answer of this question.
 To state our result, we recall monodromy birationally minimal classes, which
 is introduced in \cite[Definition 1.13]{Amerik2015}.
\begin{definition}[Monodromy Birationally Minimal Class]
Let $X$ be an irreducible symplectic manifold.
A cohomology class $e$ of $H^{1,1}(X,\RR)\cap H^2 (X,\QQ)$ is said to
be Monodromy birationally minimal if there exists an element $\gamma$ of $\MON{X}$
such that $\gamma(e)^{\perp}\cap \overline{\KEH{X}}$ is an open set of $\gamma(e)^{\perp}$.
We denote by $\MBM{X}$ the set of Monodromy birationally minimal classes of $X$.
\end{definition}
\begin{remark}
If $X$ is a $K3$ surface, then
$$
 \MBM{X} = \{ e \in H^{1,1}(X,\RR) \cap H^{2}(X,\ZZ)| \langle e,e \rangle = -2\}.
$$
\end{remark}
By the above remark, we can restate the structure of the nef cone of a $K3$ surface in the
following form. If $\MBM{X} = \emptyset$, then $\NEF{X}$ coincides with
$\overline{\POSNS{X}}$. If
$\MBM{X} \ne \emptyset$, there exists a subset $\mathscr{N}(X)$ of $\MBM{X}$ such that
$$
 \NEF{X} = \{x \in \NSR{X} | \forall e \in \mathscr{N}(X), x.e \ge 0\}
$$
Now we state the first application.
\begin{theorem}\label{main_1}
 Let $X$ be a projective irreducible symplectic manifold whose Picard number
 is greater than two. 
 If $\MBM{X} = \emptyset$, then
 $\NEF{X}$ coincides with $\overline{\POSNS{X}}$. If $\MBM{X} \ne \emptyset$,
 there exists a subset $\mathscr{N}(X)$ of $\MBM{X}$ such that
$$
 \NEF{X} = \{x \in \NSR{X} | \forall e \in \mathscr{N}(X), q_X (x,e) \ge 0\}
$$
 where $q_X$ is the Beauville-Bogomolov quadratic form.
\end{theorem}
 If $X$ is a $K3$ surface, $\RC{X}$ coincides with the set of smooth rational curves. 
 The author would like to ask the following question.
\begin{problem}
 In Theorem \ref{main_1}, does $\RC{X}$ coincide with the set of rational cohomology classes
 corresponding to smooth rational curves in $X$?  
\end{problem}
The second application concerns with the rank of a subgroup of an automorphism group 
of an irreducible symplectic manifold
which fixes  a line bundle. We recall the definition of an almost abelian group
according to  \cite{MR2296437}.
\begin{definition}\label{almost_abelian_definition}
  A group $G$ is said to be almost abelian of rank $r$ if $G$ has a normal subgroup $G^{(0)}$ 
  such that $|G : G^{(0)}| < \infty $ and 
  $G^{(0)}$ sits
  in the following exact sequence.
$$
 1 \to K \to G^{(0)} \to \ZZ^r \to 0
$$
 where $K$ is a finite group.
\end{definition}
\begin{theorem}\label{main_2}
 Let $X$ be an irreducible symplectic manifold and 
 $L$ an isotropic nef line bundle with respect to 
 Beauville-Bogomolov quadratic form.
 We define the subset $\MBM{X}^{\circ}$ of $\MBM{X}$ by
\[
  \MBM{X}^{\circ} := 
 \{
 e \in \MBM{X} | \mbox{\(e^{\perp} \cap \NEF{X} \) is an open set of  \( e^{\perp} \)}
 \}.
\]
 Let $\WR$ be a sub linear space in generated by $c_1 (L)$ and 
 $c_1 (L)^{\perp} \cap \MBM{X}^{\circ}$. 
 We denote by $\AUT{X,L}$ the subgroup of $\AUT{X}$ defined by
$$
 \AUT{X,L} := \{g \in \AUT{X} | g^*L \cong L\}
$$
 Then $\AUT{X,L}$ is almost abelian of
 rank  $\dim \NSR{X} - \dim \WR - 1$.
\end{theorem}
\begin{remark}
 In \cite[\S 3 (3.1)]{MR3165023}, Nikulin showes the similar formula for an automorphism
 of an elliptic $K3$ which preserves a fibration.
\end{remark}
\begin{remark}
 Let $X$ be a $K3$ surface. Assume that $X$ admits an elliptic fibration $\pi : X \to \mathbb{P}^1$.
 We denote by $L$ the pull back of the tautological bundle of $\mathbb{P}^1$.
 For a point $t$ of $\mathbb{P}^1$, we let $n_t$ be the number of irreducible components
 of the fibre at $t$. 
 In this case,
 $c_1(L)^{\perp}\cap \MBM{X}^{\circ}$ consists of irreducible components
 of reducible singular fibres of $\pi$. 
 Since cohomology classes of irreducible components of a reducible singular fibre has only one relation
 in $H^2 (X,\ZZ)$,
  $\dim \WR = 1 + \sum_{t \in \mathbb{P}^1}(n_t - 1) $
 and Shioda-Tate formula in \cite{MR0429918}, \cite{MR1202625} and \cite{MR1081832}
 asserts that the rank of Mordell-Weil group of $\pi$
 coincides with $\dim \NSR{X} - \dim \WR - 1$. Since Mordell-Weil group of $\pi$
 can be considered as a subgroup of $\AUT{X,L}$, 
 Theorem \ref{main_2} can be considered as a generalization of Shioda-Tate formula.
\end{remark}
 By Theorem \ref{main_2}, the rank of $\AUT{X,L}$ is less than or equal $\dim H^2 (X,\RR) - 2$. 
 It is a natural question whether this bound is sharp. 
 The third application is that 
 the bound is attained after deforming the pair $(X,L)$.
 \begin{definition}\label{Deformation_Equivalent}
  Let $X$ and $X'$ be  compact K\"ahler manifolds.
  We also let $L$ and $L'$ be
  line bundles on $X$ and $X'$, respectively.
  Two pairs $(X,L)$ and $(X',L')$ are deformation equivalent
  if there exists a smooth morphism $\pi : \XF \to B$ over an analytic space $B$
  and a line bundle $\LF$ on $\XF$
  which has the following two properties:
\begin{itemize}
 \item[(1)] There exist two points $p$ and $p'$ of $B$ such that $\iota : \pi^{-1}(p) \cong X$
	    and $\iota' : \pi^{-1}(p')\cong X'$, respectively.

 \item[(2)] The restriction of $\LF$ to $\pi^{-1}(p)$ 
	    is isomorphic to $L$ via $\iota$ and 
	    the restriction of $\LF$ to $\pi^{-1}(p')$ 
	    is isomorphic to $L'$ via $\iota'$.
\end{itemize}
 \end{definition}
\begin{theorem}\label{main_3}
 Let $X$ be an irreducible symplectic manifold whose second Betti number is greater than five and
 $L$ an isotropic line bundle with respect to Beauville-Bogomolov quadratic form.
 Then there exists an irreducible symplectic manifold $X'$ and a line bundle $L'$
 on $X'$ such that
 $(X,L)$ is deformation equivalent to $(X,L)$ in the sense of Definition \ref{Deformation_Equivalent}
 and
 the rank of $\AUT{X',L'}$ is equal to $\dim H^2 (X',\RR) - 2$.
\end{theorem}
 This note is organized as follows.
 In section 2, We will construct special elements of $O(H^2 (X,\ZZ))$,
 which is a key of the proof of Theorems \ref{main_1}, \ref{main_2} and \ref{main_3}. 
 In sections 3,4, and 5, we give a proof of Theorem 
 \ref{main_1}, \ref{main_2} and \ref{main_3}, respectively.
%
%

%%%_. 

\section*{Acknowledgement} The author would like to express his gratitude for
Professors Shigeru Mukai, Yoshinori Namikawa, Kieji Oguiso and Hisanori Ohashi.

%%%_.
\section{Construction of elements of the monodromy group}

We recall a standard properties of isometry group of a lattice due to the step 2
of the proof of \cite[Proposition 3.2]{MR3158701}.
\begin{lemma}\label{Isometry_finite}
 Let $\LL$ be a lattice and $\LL'$ a sublattice of $\LL$ such that
 $|\LL:\LL'| < \infty$. We also let $O(\LL)$ and $O(\LL')$ be
 isometry groups of $\LL$ and $\LL'$, respectively.
 The groups $O(\LL)$ and $O(\LL')$ can be considered as  subgroups of $O(\LL\otimes_{\ZZ}\QQ)$.
 Moreover
 $$|O(\LL'): O(\LL)\cap O(\LL')|$$ 
 is finite.
\end{lemma}
%
%
%
%%%_ ,
\begin{proof}
 Since $|\LL:\LL'| < \infty$, there exists a positive integer $N$
 such that $\LL \subset \frac{1}{N}\LL'$. Since $O(\LL')$ preserves $\frac{1}{N}\LL'$,
 $O(\LL')$ acts on $\frac{1}{N}\LL'/\LL'$. Then $O(\LL)\cap O(\LL')$ is the stabilizer
 group of $\LL/\LL'$. Since $\LL/\LL'$ has only finitely many element,
 we are done. 
\end{proof}
%
%%%_.
\begin{prop}\label{automorphism_of_hyperbolic_lattice}
 Let $\LL$ be a lattice of rank $n$ whose index is $(1,n-1)$.
 Assume that $\LL$ contains an isotropic element $\ell$.
 Let $W$ be a negative definite sub lattice contained in  $\ell^{\perp}$.
 Assume that $n - \RK{W} > 2$.
\begin{itemize}
 \item[(1)]
 We define the subgroup $\BG $ of the isometry group $O(\LL)$ of $\LL$ by
$$
 \BG := \{g\in O(\LL) | g(\ell) = \ell, \forall w \in W, g(w) = w
 \}
$$
 Then $\BG $ contains a subgroup $\BG_0$
 which is isomorphic to $\ZZ^{n-\mathrm{rank}W - 2}$ and $|\BG : \BG_0 | < \infty$.
\item[(2)]
 We denote by $\LR$ the linear space $\LL\otimes_{\ZZ}\RR$ and define
 a positive cone  $\mathscr{C}(\LR)$ by
\[
 \mathscr{C}(\LR) := \{
 x \in \LR | \langle x,x \rangle > 0
\}.
\]
 For every element $g$ of $\Gamma_0$ and every element $x$ of $\POS{\LR}$, 
$$
 \lim_{m\to \infty} g^{m}x = \ell \quad \mbox{in} \quad \mathbb{P}(\LR).
$$
\end{itemize}
\end{prop} 
%
%%%_
\begin{proof}
(1) Let $W^{\perp}$ be the orthogonal lattice of $W$. The restriction
\[
 \BG \ni g \mapsto g|_{W^{\perp}} \in O(W^{\perp})
\]
 is injective, because $g$ acts on $W$ trivially.
 We identify $\BG$ and its image.
 Since the index of $W^{\perp}$ is $(1,n - \RK{W} - 1)$,
 by \cite[Proposition 2.9]{MR2296437},  $\BG$ contains
 a subgroup $\BG'_0$ 
 which is isomorphic to \( \ZZ^{m} \), \((0 \le m \le n- \RK{W} - 2)\).
 Moreover $| \BG : \BG'_0 | < \infty$.
 Let us assume that
 we have a subgroup $\BG_0$ of $\BG$ which is isomorphic
 to $\ZZ^{n - \RK{W} - 2}$. 
 Since $| \BG_0 : \BG_0 \cap \BG'_0 | \le |\BG : \BG'_0 | < \infty$, $m \ge n - \RK{W} - 2$
 and we are done. Hence
 we will construct a subgroup \( \BG_0 \)
 which is isomorphic to $\ZZ^{n - \RK{W} - 2}$.
 Let us consider the projection $r : \ell^{\perp} \cap \LL \to \ell^{\perp} \cap \LL / \ZZ \ell$.
 Replacing $W$ by its saturation in $\LL$,
 we may assume that $W$ is primitive. Since $W$ is negative definite,
 $W \cong r(W)$. Moreover, $r(W)$ is primitive.
 We choose $\RK{W}$ elements $\{u_{1},\ldots , u_{\RK{W}} \}$ 
 of \(W\)
 such that the set of the residue classes $\{\bar{u}_1,\ldots ,\bar{u}_{\RK{W}} \}$ 
 forms a generator of $r(W)$. Since $r(W)$ is primitive, we have $n - \RK{W} - 2$ 
 elements $\{ u_{\RK{W} + 1}, \ldots , u_{n-2} \}$ such that
 the residue classes $\{ \bar{u}_{1}, \ldots , \bar{u}_{n-2} \}$ 
 forms a generator of $\ell^{\perp}/ \ZZ \ell$.
 Then
 $\{\ell , u_1 , \ldots, u_{n-2}\}$ forms a generator of $\ell^{\perp}$.
 Since $\ell^{\perp}$ is a primitive sub lattice of $\LL$, we have
 an element $\ell'$ of $\LL$ such that
 $\{\ell , u_1 , \ldots , u_{n-2} , \ell'\}$ forms a generator of $\LL$.
 The gram matrix $G_{\LL}$ of the bilinear form of $\LL$ 
 with respect to the basis
 $\{\ell , u_1 , \ldots , u_{n-2} , \ell' \}$ can be described as follows;
$$ G_{\LL} = 
\begin{pmatrix}
 0 & 0 & a \\
 0 & A & b\\
 a & {}^t b & c
\end{pmatrix}
$$
 where $A$ is a negative definite symmetric matrix and $a$ is a nonzero integer.
 We put $d = \det A$.
 For an integer $i$ with $\RK{W} + 1 \leq i \leq n-2$,
 we define a  $n-2$ row vector $\gamma_i$ by
$$
 (\mbox{The $j$-th column of $\gamma_i$}) =
\begin{cases}
 d & j=i \\
 0 & j\ne i
\end{cases}
$$
 Let $\gamma $ be a linear combination of $\gamma_i$, $(\RK{W} + 1\le i \le n-2)$.
 We define the matrix $T(\gamma)$ by
$$ T(\gamma)=
\begin{pmatrix}
 1 & -2 \gamma  & -2a\gamma (A^{-1}) {}^t\gamma -2\gamma (A^{-1})b\\
 0 & E      & 2a (A^{-1}) {}^t \gamma \\
 0 & 0      &  1
\end{pmatrix},
$$
 where $E$ is the $(n-2)\times (n-2)$ identity matrix.
 Since ${}^tT(\gamma)G_{\LL} T(\gamma) = G_{\LL}$,
 there exists an element $g(\gamma)$ of $O(\LL)$ whose matrix of representation 
 with respect to the basis $\{\ell, u_1 , \ldots, u_{n-2}, \ell'\}$ coincides with $T(\gamma)$.
 By definition
 \begin{eqnarray}
 g(\gamma) (\ell) &=& \ell\label{154051_24Aug18} \\
 g(\gamma) (u_i)  &=& 
 u_i -  2a_i \ell \quad (1 \leq i \leq n-2)\label{154059_24Aug18} \\
   g(\gamma) (\ell')  &=& 
 \ell' + 2a\sum_{i=1}^{n-2}b_i u_i  - (2a (\gamma (A^{-1}){}^t \gamma) + 2 \gamma (A^{-1})b) \ell\label{154108_24Aug18} 
\end{eqnarray} 
 where $a_i$ is the $i$-th column of $\gamma$ and
       $b_i$ is  the $i$-th row of $(A^{-1}){}^t \gamma$.
 By definition, $a_i = 0$, $(1 \le i \le \RK{W})$. 
 Hence $g(\gamma)$ is  an element of $\Gamma $. Moreover
$$
g(\gamma + \gamma') = g(\gamma)g(\gamma')=
g(\gamma')g(\gamma)
$$
 for all linear combinations $\gamma$ and $\gamma'$ of $\gamma_i$, $(\RK{W} + 1\le i \leq n -2)$. 
 By definition, 
 $g(\gamma) = E$ if and only if $\gamma = 0$.
 We define the subgroup $\BG_{0}$ of $\BG$ generated by 
 $ g(\gamma_i)$, $(\RK{W} + 1 \leq i \leq n - 2)$.
 By construction,
 \(\BG_0\) is
 isomorphic to $\ZZ^{n - \RK{W}-2}$ and we are done.
\newline 
(2) We will use the same notation  in the proof of part (1).
 For an element $x$ of $\POS{\LR}$, we have the following expression.
\[
 x = \alpha_0 \ell + \sum_{i=1}^{n-2} \alpha_i u_i + \beta \ell' .
\]
 By the equations (\ref{154051_24Aug18}), (\ref{154059_24Aug18}) and (\ref{154108_24Aug18}),
 we have
\begin{eqnarray*}
 g(m\gamma)(x) 
 &=& \left(\alpha_0 - 
 \sum_{i=1}^{n-2} 2 m \alpha_i a_i
 -\beta (2am^2 (\gamma (A^{-1}){}^t \gamma) + 2m \gamma (A^{-1})b) 
 \right)\ell  \\
 &+& \sum_{i=1}^{n-2} (\alpha_i  + 2\beta amb_i )u_i \\
 &+& \beta \ell' 
\end{eqnarray*}
 Since $\langle x,x \rangle > 0$ and the index of the induced bilinear form
 on $\ell^{\perp}$ is $(0,0,n-2)$, $\beta \ne 0$. 
 Hence the order of growth of the coefficient of $\ell$ is $m^2$, 
 while the order of growth of other coefficients are at most $m$.
 This implies that
$$ 
 \lim_{m \to \infty} g(m\gamma ) (x) = \ell \; \mbox{in} \; \mathbb{P}(\LR)
$$
 and  we are done.
\end{proof}
%
%
%%%_
\begin{corollary}\label{subgroup_of_monodromy}
 Let $X$ be a projective symplectic manifold. Assume that
 there exists an element $\ell$ of $\NS{X}$ which is isotropic with respect to 
 Beauville-Bogomolov quadratic form.
 We denote by $\MON{X}$ the monodromy group of $X$
 and by $n$ the Picard number of $X$. 
 Let $W$ be a negative definite sublattice of $\NS{X}$ which is contained in $\ell^{\perp}$.
 Assume that $n - \RK{W} > 2 $.
 Then $\MON{X}$ contains a subgroup $\Gamma$ which has the following four properties:
\begin{itemize}
 \item[(1)] $\Gamma$  is isomorphic to $\ZZ^{n-\mathrm{rank}W-2}$;
 \item[(2)] The action of $\Gamma$ respects the Hodge structure of $H^2 (X,\ZZ)$ and 
	    $\Gamma$ acts on the transcendental lattice of $H^2 (X,\ZZ)$ trivially;
 \item[(3)] For every element $g$ of $\Gamma$, $g(\ell) = \ell$ and
 $g(w)=w$ for all elements of $W$ and;
 \item[(4)] For every element $g$ of $\Gamma$ and every element $x$ of $\POSNS{X}$,
$$
 \lim_{m\to \infty} g^m x = \ell \;
 \mbox{in} \;
 \mathbb{P}(\NSR{X}).
$$
\end{itemize}
\end{corollary}
%
%
%
%%%_
\begin{proof}
 By Proposition \ref{automorphism_of_hyperbolic_lattice},
 we have a subgroup $\bar{\Gamma}$ of $O(\NS{X})$ which has the following three
 properties:
\begin{itemize}
 \item[(1)] $\bar{\Gamma}$  is isomorphic to $\ZZ^{n - \mathrm{rank}W -2}$;
 \item[(2)] For every element $g$ of $\bar{\Gamma}$, $g(\ell) = \ell$ and $g(w)=w$ for
	    all elements of $W$ and;
 \item[(3)] For every element $g$ of $\bar{\Gamma}$ and every element $x$ of $\POSNS{X}$,
$$
 \lim_{m\to \infty} g^m x = \ell \quad
 \mbox{in} \quad
 \mathbb{P}(\NSR{X}).
$$
\end{itemize} 
Let $\NS{X}^{\perp}$ be the orthogonal lattice of $\NS{X}$ in $H^2 (X,\ZZ)$ 
with respect to Beauville-Bogomolov quadratic form. We recall $\NS{X}^{\perp}$ is
nothing but the transcendental lattice. 
We define a subgroup $\Gamma'$ of $O(\NS{X}\oplus \NS{X}^{\perp})$ by
$$
 \Gamma' := \left\{
 g \oplus \mathrm{id}_{\NS{X}^{\perp}} | g\in \bar{\Gamma}
\right\}
$$
 Since $|H^2 (X,\ZZ): \NS{X}\oplus \NS{X}^{\perp}| < \infty$, 
 $$|O(\NS{X}\oplus \NS{X}^{\perp}): O(\NS{X}\oplus \NS{X}^{\perp})\cap O(H^2 (X,\ZZ))| < \infty$$ 
 by Lemma \ref{automorphism_of_hyperbolic_lattice}.
 By the definition, the action of
 $\Gamma' \cap O(H^2 (X,\ZZ))$ respects the Hodge structure of $H^2 (X,\ZZ)$ and 
 $\Gamma' \cap O(H^2 (X,\ZZ))$
 acts on the transcendental lattice of $H^2 (X,\ZZ)$
 trivially.
By \cite[Theorem 7.2]{MR3161308} and \cite[Theorem 2.6]{Amerik2017},
 $$|O(H^2 (X,\ZZ)): \MON{X} | < \infty .$$
 Hence $|\Gamma' : \Gamma' \cap \MON{X}| < \infty$. 
 If we define $\Gamma$ by $\Gamma' \cap \MON{X}$, 
 we are done. 
\end{proof}

\section{Proof of Theorem \ref{main_1}}

%%%_
Before starting to prove
Theorem \ref{main_1}, we prepare Lemma \ref{Nef_cone_inclusion}.

\begin{lemma}\label{Nef_cone_inclusion}
 Let $X$ be an irreducible symplectic manifold. Assume that
 the nef cone $\NEF{X}$ contains an open set $U$ of $\BPOSNS{X} $ and
 $\MBM{X} \ne \emptyset$. Then $\BPOSNS{X} \cap \NS{X} \ne \{0\}$.
\end{lemma}
%
%%%_
%
\begin{proof}
 Let $e$ be an element of  $\MBM{X}$ such that $e^{\perp} \cap \NEF{X}$ 
 is an open set of $e^{\perp}$.
 We choose a $2$-plane $H$ in $\NSR{X}$ as $H$ contains $e$, $H\cap U \ne \emptyset$
 and $H$ is defined over $\NSQ{X}$.
 The restriction $\NEF{X} \cap H$ is generated by two rays $\ell_1$ and $\ell_2$.
 Since $H\cap U \ne \emptyset$,
 we may assume that $q_{X} (\ell_1)= 0$, where 
 $q_X$ is the Beauville-Bogomolov quadratic form of $X$.
 Let $\pi : \XF \to \DEF{X}$ be a Kuranishi family of $X$.
 We define the subset $\Omega $ in $\PP{H^2(X,\CC)}$ by
\[
 \Omega := \{ x \in \PP{ H^2(X,\CC)} | q_X(x) = 0, q_X(x+\bar{x}) > 0\}.
\]
 By \cite[Th\'{e}or\`{e}me 5]{Beauville},
 we have a morphism $p : \DEF{X} \to \Omega$, which is locally isomorphic.
 We choose a point $t$ of $\DEF{X}$ such that $(p(t)^{\perp} \cap H^2 (X,\ZZ))\otimes_{\ZZ}\RR = H$.
% We note that such  points form a complement of a union of countably hyperplanes of $H^2 (X,\CC)$.
 Let $\XF_t$ be the fibre of $\pi$ at $t$. Then $\NSR{\XF_t} = H$.
 We have an induced diffeomorphism $\iota : \XF_t \cong X$.
 By \cite[Corollary 5.13]{Amerik2015}, $\MBM{\XF_t} = \iota^* (\MBM{X} \cap H)$.
 Hence $\NEF{\XF_t} \supset \iota^* (\NEF{X} \cap H)$ and $\iota^* (\ell_1)$ is
 a ray of $\NEF{\XF_t}$. Since $e \in H$,
 $\MBM{\XF_t} \ne \emptyset$. By \cite[Theorem 1.19]{Amerik2015}, $\NEF{\XF_t} \ne \overline{\POSNS{\XF_t}}$.
 By \cite[Theorem 1.3 (1)]{Oguiso2014}, two rays of $\NEF{\XF_t}$ are rational, especially $\iota^* (\ell_1)$
 is rational.
  Since $\iota^*$ preserves rationalities,
 we are done.
\end{proof}
%
%%%_
%
\begin{proof}[Proof of Theorem \ref{main_1}] 
If $\MBM{X} = \emptyset$, $\NEF{X} = \overline{\POSNS{X}}$ by \cite[Theorem 1.19]{Amerik2015},
and we are done. 
Assume that $\MBM{X} \ne \emptyset$. We choose a K\"{a}hler class $\kappa$ of $H^2 (X,\RR)$.
We define the subset $\RC{X}$ of $\MBM{X}$ by
$$
 \RC{X} := \{
 e \in \MBM{X} | \mbox{
 $e^{\perp} \cap \NEF{X}$ is an open set of $e^{\perp}$, \(q_X(e,\kappa) > 0\)
  }
 \}.
$$
We also define
the cone $D$ in $\NSR{X}$ by
$$
 D := \{ x \in \NSR{X} | \forall e \in \RC{X}, q_X (e,x) \ge 0
 \}.
$$
If $\NEF{X} = D$, we are done. 
We derive a contradiction
assuming $\NEF{X} \ne D$. By \cite[Theorem 1.19]{Amerik2015},
$D \cap \overline{\POSNS{X}} = \NEF{X}$. Hence $D$ 
contains an element $x$ of $\NSR{X}$ such that $q_X (x) < 0$.
This implies that $\NEF{X} \cap \BPOSNS{X}$
contains an open set of $\BPOSNS{X}$.
Since $\MBM{X} \ne \emptyset$, 
$\BPOSNS{X} \cap \NS{X} \ne \emptyset$ by
Proposition\ref{Nef_cone_inclusion}.
The boundary $\BPOSNS{X}$ is defined by a rational quadratic form. 
Hence $\BPOSNS{X}\cap \NS{X}$
forms a dense subset of $\BPOSNS{X}$. 
The intersection
$\BPOSNS{X} \cap \NEF{X}$  contains an open set of  $\BPOSNS{X}$ and
we have a nonzero element $\ell$ of $\BPOSNS{X}\cap \NEF{X} \cap \NS{X}$ 
such that
$$ \ell^{\perp} \cap \MBM{X}  =  \emptyset .$$ 
Let $\Gamma$ be a subgroup of $\MON{X}$ obtained by Corollary \ref{subgroup_of_monodromy}.
For an element $g$ of $\Gamma$,
  by\cite[Theorem 1.19]{Amerik2015} and \cite[Lemma 5.7]{Markman_2011}, $g(\AMP{X})$ is an connected component
of $\POSNS{X} \setminus \bigcup_{e \in \MBM{X}} e^{\perp}$. Hence if
$g(\AMP{X}) \ne \AMP{X}$, then there exists an element $e$ of $\MBM{X}$ such that
the hyperplane $e^{\perp}$ separates $\AMP{X}$ and $g(\AMP{X})$, that is,
$$
 \AMP{X} \subset e^{> 0}, \; g(\AMP{X}) \subset e^{< 0}
$$
where $e^{>0} := \{x \in \NSR{X} | q_X (x,e) > 0 \}$.
Since $g(\ell) = \ell$, $e$ should be an element of $\ell^{\perp} \cap \MBM{X}$.
By the choice of $\ell$, such a class does not exist and $g(\AMP{X}) = \AMP{X}$.
By Theorem \ref{Global_Torelli_Theorem}, there exists an automorphism $\Phi$  of $X$
such that $\Phi^* = g$. 
By Proposition \ref{subgroup_of_monodromy},
 $\lim_{m\to \infty}g^m x = \ell$ in $\mathbb{P}(\NSR{X})$ 
for all $x \in \POSNS{X}$.
Hence, for every element $x$ of $\POSNS{X}$,
there exists a positive integer $N$ such that $(\Phi^N)^*x \in \AMP{X}$.
This implies that $\NEF{X} = \overline{\POSNS{X}}$.
By \cite[Theoorem 1.19]{Amerik2015},  $\MBM{X} = \emptyset$. This contradicts
the first assumption that $\MBM{X} \ne \emptyset$.
\end{proof}

\section{Proof of Theorem \ref{main_2}}
\begin{proof}
 First we will prove that $\AUT{X,L}$ is an almost abelian group and 
its rank is at most 
 \(\dim \NSR{X} - \dim \WR - 1 \). 
 Let $\Gamma $ be the image of the natural representation
 $\rho : \AUT{X,L} \to O(\NS{X})$.
 By \cite[Corollary 2.7]{MR2406267}, the kernel of $\rho$ is finite. 
 Hence it is enough to prove that $\Gamma $ is an almost abelian group
 of rank at most $\dim \NSR{X} - \dim \WR - 1$ by \cite[Proposition 9.3 (2)]{MR2406267}.
 Let us consider the natural projection 
\[
r :c_1 (L)^{\perp} \to c_1 (L)^{\perp}  / \ZZ c_1 (L). 
\]
 We define the lattice $\bar{W}$ by 
\[
 \bar{W} := r (\WR) \cap \left(c_1 (L)^{\perp}\cap \NS{X} / \ZZ c_1 (L)\right). 
\]
 We choose elements $\{e_1 , \ldots, e_k\}$ of $c_1 (L)^{\perp}\cap \MBM{X}^{\circ}$
  as their residue classes give a  generator of $\bar{W}$.
 We note that $k = \RK{\bar{W}} = \dim \WR - 1$.
 Let $W$ be the sub lattice of $\NS{X}$ generated by $\{e_1 , \ldots , e_k\}$.
 Then there exists a natural isomorphism $W \cong \bar{W}$.
 Since the induced bilinear form on $c_1(L)^{\perp}/\RR c_1(L)$ is negative definite,
 $\bar{W}$ is negative definite and hence $W$ also is. Let $W^{\perp}$ be the orthogonal lattice of $W$
 with respect to the Beauville-Bogomolov quadratic form.
 Since $\Gamma$ preserves $c_1 (L)$ and $c_1 (L)^{\perp} \cap \MBM{X}^{\circ}$, 
 $\Gamma$ preserves $W$ and $W^{\perp}$. We consider the following homomorphism
\[
 \mu_1 :  \Gamma \ni g \to g|_{W}\oplus g|_{W^{\perp}} \in O(W)\oplus O(W^{\perp}).
\]
 and the projection $\mu_2 : O(W) \oplus O(W^{\perp}) \to O(W^{\perp})$.
 Since $| \NS{X} : W \oplus W^{\perp} < \infty|$, $\mu_1$ is injective.
 Since $W$ is negative definite, $O(W)$  is finite and the kernel of $\mu_2$ is finite.
 Hence the kernel of $\Gamma \to \mu_2 \circ \mu_1 (\Gamma)$ is finite.
 Therefore it is enough to prove that $ \mu_2 \circ \mu_1 (\Gamma)$
 is almost abelian of rank at most 
 \[
  \dim \NSR{X} - \dim \WR - 1  = 
  \dim \NSR{X} - \RK{W} -2 
 \]
 by \cite[Proposition 9.3 (2)]{MR2406267}.
 Since $W^{\perp}$ is a lattice whose index is $(1, n - \RK{W} - 1)$ and
 $\Gamma $ preserves $c_1 (L)$, $ \mu_2 \circ \mu_1 (\Gamma)$ 
 is an almost abelian group whose rank is at most $n - \RK{W} - 2$ by
 \cite[Proposition 2.9]{MR2296437} and we are done.

 Next
 we will prove that
\[
 \RK{\AUT{X,L}} 
 =
\dim \NSR{X} - \RK{W}  - 2 .
\]
 Since $W$ is negative definite and contained in $c_1 (L)^{\perp}$,
 we have a subgroup $\Gamma_0$ of $\MON{X}$ by Corollary \ref{subgroup_of_monodromy}.
 We note that $\Gamma_0$ is isomorphic to
 $\ZZ^{\dim \NSR{X} - \RK{W} - 2}$.
 Let $g$ be an element of $\Gamma_0$.
 By \cite[Theorem 1.19]{Amerik2015} and \cite[Lemma 5.17]{Markman_2011},
 $g(\AMP{X})$ coincides with a connected component of $\POSNS{X}\setminus \bigcup_{e \in \MBM{X}}e^{\perp}$
 whose closure contains $c_1 (L)$. Hence,
 if $g(\AMP{X}) \ne \AMP{X}$, there exists an element $e$ of $c_1 (L)^{\perp} \cap \MBM{X}^{\circ}$ such that
 $g(\AMP{X}) \subset e^{>0}$ and $\AMP{X} \subset e^{<0}$. 
 By the definition of $W$ and Corollary \ref{subgroup_of_monodromy}, $g$ fixes
 $c_1 (L)$ and all elements of
 $c_1 (L)^{\perp} \cap \MBM{X}^{\circ}$. 
 Hence there are no such elements in $c_1 (L)^{\perp} \cap \MBM{X}^{\circ}$.
 Therefore $g(\AMP{X}) = \AMP{X}$ and there exists an element $\Phi$ of
 $\AUT{X,L}$ such that $\Phi^* = g$ by Theorem \ref{Global_Torelli_Theorem}.
 This implies that $\Gamma_0$ is a subgroup of $\AUT{X,L}$ and
 the rank of $\AUT{X,L}$ coincides with $\dim \NSR{X} - \RK{W}  - 2$.
\end{proof}

\section{Proof of Theorem \ref{main_3}}
\begin{lemma}\label{density_of_period}
 Let $\BLL$ be a lattice whose index is $(2, \RK{\BLL}  - 2)$. 
We fix a positive integer $N$ and define
$$
 \BLL_N := \{ x \in \BLL | -N < \langle x,x \rangle < 0\}
$$
 We denote by $\GR{2,\BLLR}$ the open set of Grassmanian $\mathrm{Gr}(2,\BLLR)$ 
 which consists of positive $2$-planes
 in $\BLLR$.
 Let
 $\GR{2,\BLLR}^{\circ}$ be a subset of $\GR{2,\BLLR}$ defined by
$$
\GR{2,\BLLR}^{\circ} :=
 \{
  \sigma \in \GR{2,\BLLR}
 | \forall x \in \BLL_N, \sigma \not\subset x^{\perp}
 \}
$$
 Then $\GR{2,\BLLR}^{\circ}$
 is a non empty open in $\GR{2,\BLLR}$.
\end{lemma}
\begin{proof}
 If we choose two positive vectors $v,w$ 
 in $\BLLR \setminus \bigcup_{x \in \BLL_N}x^{\perp}$,
 the positive $2$-plane $\sigma$ spanned by $v$ and $w$ is an element of 
 $\GR{2,\BLLR}^{\circ}$ and $\GR{2,\BLLR}^{\circ}$ is nonempty.
 Let $\sigma_0$ be an element of $\GR{2,\BLLR}^{\circ}$ and $V_{\sigma_0}$
 a small closed neighborhood of $\sigma_0$ in $\GR{2,\BLLR}$ with respect to
 an Euclidean topology of $\GR{2,\BLLR}$. We define the subset $U_{\sigma_0}$ of 
 $V_{\sigma_0}\times \BLLR$ by
\[
 U_{\sigma_0} := 
 \{
 (\sigma, x) \in V_{\sigma_0}\times \BLLR | 
-N \leq \langle x,x \rangle \leq 0, x \in \sigma^{\perp}
 \}
\]
 We denote by $p_1$ the first projection $U_{\sigma_0} \to V_{\sigma_0}$
 and by $p_2$ the second projection $U_{\sigma_0} \to \BLLR$.
 Then 
\[
 p_2 (U_{\sigma_0}) \supset \bigcup_{\sigma \in V_{\sigma_0}} (\sigma^{\perp} \cap \BLL_N)
\]
 Since the index of the bilinear form on $\BLLR$ is $(2,\dim \BLLR - 2)$, 
 each fibre of $p_1$ is homeomorphic to a ball of dimension $\dim \BLLR - 2$.
 By definition, $V_{\sigma_0}$ is compact and
 hence $U_{\sigma_0}$ is also compact.
 This implies that the set
\[
  \bigcup_{\sigma \in V_{\sigma_0}} (\sigma^{\perp} \cap \BLL_N)
\]
 has only finite elements. Hence $V_{\sigma_0} \cap \GR{2, \BLLR}^{\circ}$ is
 an open subset of $V_{\sigma_0}$ and we are done.
\end{proof}
\begin{corollary}\label{density_of_period_2}
 Let $\Lambda$ be a lattice whose index is $(3,\RK{\Lambda} -3)$. We fix a positive
 integer $N$. Assume that $\Lambda$ has an isotropic element $\ell$.
 Let $\Lambda_N$ be the subset of $\Lambda$ defined by
$$
 \Lambda_N :=
\{
 x \in \Lambda | -N < \langle x,x \rangle < 0
\}.
$$
 We denote by $\GR{2,\ell^{\perp}}$ the open set of Grassmanian $\operatorname{Gr}(2,\ell^{\perp})$ which
 consists of positive $2$-planes in $\ell^{\perp}$. 
 Then the subset $\GR{2,\ell^{\perp}}^{\circ}$ defined by
$$
 \GR{2, \ell^{\perp}}^{\circ} :=
 \{
 \sigma \in \GR{2,\ell^{\perp}}
 | \forall x \in \Lambda_N \cap \ell^{\perp}, \sigma \not\subset x^{\perp}
 \}
$$
 is open.
\end{corollary}
\begin{proof}
 We denote by $\BLL$ the quotient lattice $\Lambda \cap \ell^{\perp}/\ZZ \ell$.
 The symbols $\BLLR$, $\BLL_{N}$, $\GR{2,\BLLR}$ and $\GR{2,\BLLR}^{\circ}$ represent
 the same objects in Lemma \ref{density_of_period}.
 Let us consider the projection $\pi : \ell^{\perp} \to \BLLR = \ell^{\perp}/\RR \ell$.
 Since $\pi$ respect the bilinear forms on $\ell^{\perp}$ and $\BLLR$,
 we have the induced morphism $\pi' : \GR{2,\ell^{\perp}} \to \GR{2,\BLLR}$. 
 By definition, $(\pi')^{-1}(\GR{2,\BLLR}^{\circ}) = \GR{2,\ell^{\perp}}^{\circ}$.
 By Lemma \ref{density_of_period}, $\GR{2,\BLLR}^{\circ}$ is open and we are done.
\end{proof}
 We recall the definition of 
 marked irreducible symplectic manifolds, their moduli and the global 
period map.
\begin{definition}\label{marked_irreducible}
 Let $\Lambda$ be a lattice whose index is $(3, \RK{\LL} - 3)$.
 A marked irreducible symplectic manifold $(X,\varphi)$ 
 is a pair of an irreducible symplectic manifold $X$
 and an isometry $\varphi : H^2 (X,\ZZ) \to \Lambda$.
 Two marked irreducible symplectic manifold $(X,\varphi)$ and $(X',\varphi')$
 are isomorphic if there exists an isomorphism $\Phi : X \cong X'$
 such that $\varphi' = \varphi \circ \Phi^*$, where
 $\Phi^*$ is the induced isometry $H^2 (X',\ZZ) \to H^2 (X,\ZZ)$.
 A moduli space of marked irreducible symplectic manifold $\MF{\Lambda}$
 is the set of isomorphic classes of marked irreducible symplectic manifolds.
 We define the global period map $\PF : \MF{\LL} \to \PP{\LLC} $ by
$$
 \PF : \MF{\Lambda}\ni (X,\varphi) \to \varphi (H^{2,0}(X)) \in \PP{\LLC},
$$
 where $\LLC = \Lambda\otimes_{\ZZ}\CC$.
\end{definition}
 The following lemma is well-known for specialist and we add it for readers convenience.
\begin{lemma}[{\cite[(1.18)]{MR1664696}}]\label{moduli_space_is_complex_manifold}
 The symbols $(X,\varphi)$, $\LL$, $\MF{\LL}$ and $\PF$ represent the same objects
 in Definition \ref{marked_irreducible}. We define the subset of $\PP{\LL}$ by
$$
 \PD{\LL} := \{
 x \in \PP{\LLC} | \langle x,x \rangle = 0, \langle x, \bar{x} \rangle > 0
\}
$$
 Then $\MF{\LL}$ is a complex manifold and $\PF$ is a holomorphic morphism.
 The image of $\PF$ is contained in an open set of $\PD{\LL}$.
\end{lemma}
\begin{proof}
 Let $\XF \to \DEF{X}$ be the Kuranishi family of $X$.
 By \cite[Th\'{e}or\`{e}me 5]{Beauville}, we have a holomorphic morphism
 $p_X : \DEF{X} \to \PD{\LL}$. If $(X',\varphi')$ is another marked irreducible
 symplectic manifold which is isomorphic to $(X,\varphi)$ 
 in the sense of Definition \ref{moduli_space_is_complex_manifold},
 $p_{X'} : \DEF{X'} \to \PD{\LL}$ can be patched $p_X$ by
 universality of the Kuranishi space. Hence
 $\MF{\Lambda}$ carries a structure of a complex manifold
 and $\PF$ is holomorphic. Since $p_X$ is locally isomorphic,
 the image of $\PF$ is an open set of $\PD{\LL}$.
\end{proof}
 We prove a property of fibres of the global period map.
\begin{lemma}\label{fibre_property}
 The symbols $(X,\varphi)$, $\MF{\Lambda}$, $\PF$ and $\LL$ represent the same objects in 
 Definition \ref{marked_irreducible}. Let $\MF{\Lambda}^{\circ}$ be a connected component of $\MF{\Lambda}$
 which contains $(X,\varphi)$. We denote by $t$ the point $\PF (X,\varphi)$.
 Assume that $\Lambda \cap t^{\perp}$ has an isotropic element $\ell$. Then
 there exists a marked irreducible symplectic manifold $(X',\varphi')$ such that
 $(X',\varphi') \in \MF{\LL}^{\circ}$,
 $X'$ carries a nef line bundle $L'$ with $\varphi'(c_1 (L')) = \ell$ and $\PF (X',\varphi') = t$.
\end{lemma}
\begin{proof}
 Let $\POS{X}$ be the positive cone of $H^{1,1}(X,\RR)$
 and $\MBM{X}$ the set of Monodromy birationally minimal classes.
 We choose a connected component $C$ of $\POS{X}\setminus \bigcup_{e \in \MBM{X}}e^{\perp}$
 such that $\varphi^{-1}(\ell)$ is contained in $\bar{C}$, where $\bar{C}$ is the closure
 of $C$.
 We define the subgroup $\MONHDG{X}$ by
$$
 \MONHDG{X} := \{
 \gamma \in \MON{X} | \gamma (H^{2,0}(X)) = H^{2,0}(X)
\}
$$
 By \cite[Definition 6.1 and Theorem 6.2]{Amerik2015},
 there exists an irreducible symplectic manifold $X'$, a bimeromorphic map
 $f : X \dasharrow X'$ and an element $\gamma$ of $\MONHDG{X}$
 such that $\gamma \circ f^{*}(\mathcal{K}(X')) = C$,
 where $f^*$ is the induced morphism $H^2 (X',\RR) \to H^2 (X,\RR)$ and $\mathscr{K}(X')$ is 
 the K\"{a}hler cone of $X'$.
 By definition of $\MON{X}$, $(X,\varphi)$ and
 $(X,\varphi \circ \gamma)$ belong to a same connected component of $\MF{\LL}$. 
 By \cite[Theorem 2.5]{MR1992275},
 $(X,\varphi \circ \gamma)$ and $(X',\varphi \circ \gamma \circ f^{*})$ belong to a same connected
 component of $\MF{\LL}$.
 Hence if we define $\varphi' = \varphi \circ \gamma \circ f^*$,
 $(X',\varphi')$ belongs to $\MF{X}^{\circ}$. Since $\gamma \circ f^{*}(H^{2,0}(X')) = H^{2,0}(X)$,
 $\PF (X',\varphi') = t$.
 We finish the proof of Lemma if we have proved that $X'$ carries a line bundle $L'$ 
 such that $L'$ is nef and $c_1 (L') = (\varphi')^{-1}(\ell)$.
 Since $(\varphi')^{-1}(\ell) \in H^{1,1}(X',\RR)\cap H^2 (X',\ZZ)$, 
 there exists a line bundle $L'$ on $X'$ such that $c_1 (L') = (\varphi')^{-1}(\ell)$.
 By definition, $\varphi^{-1}(\ell) \in \bar{C}$. Thus
 $(\varphi')^{-1}(\ell)$ is contained in the closure of $\mathcal{K}(X')$.
 Hence $L'$ is nef and we are done.
\end{proof}
The following Proposition is the punch line of the proof of 
Theorem \ref{main_3}.
\begin{prop}\label{Nice_Period_Model}
 Let $X$ be an irreducible symplectic manifold whose Betti number is greater 
 than five and $L$ a line bundle on $X$ with $q_X(c_1 (L))= 0$, where
 $q_X$ is 
 the Beauville-Bogomolov quadratic form.
 Then there exists an irreducible symplectic manifold $X'$ and a line bundle $L'$
 which has the following four properties:
\begin{itemize}
 \item[(1)] The pairs $(X,L)$ and $(X',L')$ are deformation equivalent
            in the sense of Definition \ref{Deformation_Equivalent};
 \item[(2)] The line bundle $L'$ is nef;
 \item[(3)] The Picard number of $X'$ is equal to $\dim H^2 (X' ,\RR) - 2$ and;
 \item[(4)] The intersection  $c_1 (L')^{\perp} \cap \MBM{X'} $ is empty.
\end{itemize}
\end{prop}
\begin{proof}
Let $\LL$ be a lattice isomorphic to $(H^2 (X,\ZZ), q_X)$, where $q_X$ is the
Beauville-Bogomolov quadratic form. The symbols $\MF{\LL}$ and $\PF$
represents the same objects in Definition \ref{marked_irreducible}.
We put $\ell = \varphi (c_1 (L))$.
Let $\PD{\LL,\ell^{\perp}}$ be a subset of $\PD{\LL}$ defined by
$$
 \PD{\LL,\ell^{\perp}} := \{ x \in \PD{\LL} | \langle x,\ell \rangle = 0
\}.
$$
We will define two subsets 
of $\PD{\LL , \ell^{\perp}}$. The first one is defined by
$$
 \PD{\LL, \ell^{\perp}}^{\mathrm{max}} 
:= \{x \in \PD{\LL, \ell^{\perp}} |\RK{x^{\perp}\cap \LL} = \RK{\LL} - 2 \} 
$$
\begin{claim}\label{density_of_higher_picard_number}
 The subset $\PD{\LL,\ell^{\circ}}^{\mathrm{max}}$ is dense.
\end{claim}
\begin{proof}
 We choose an element $t$ of $\PD{\LL,\ell^{\perp}}$. 
 There exist sequences $a_m$ and $b_m$ in $\LL_{\QQ}$ such that
$$
 \lim_{m \to \infty} a_m = \mathrm{Re}(t), \; \lim_{m \to \infty} b_m = \mathrm{Im}(t)
$$
 Since $t \in \PD{\LL, \ell^{\perp}}$, 
 $\langle \mathrm{Re}(t),\mathrm{Re}(t) \rangle > 0$,
 $\langle \mathrm{Im}(t),\mathrm{Im}(t) \rangle > 0$ 
 and
 $\langle \mathrm{Re}(t),\mathrm{Im}(t) \rangle = 0$.
 Hence we may assume that $\langle a_m , a_m \rangle > 0$ and $\langle b_m , b_m \rangle > 0$
 for all $m$. Moreover we may assume that $\lim_{m \to \infty} \langle a_m , b_m \rangle = 0$.
 We define other sequences $c_m$ and $d_m$ in $\LL_{\QQ}$ by
\begin{eqnarray*}
 c_m &=& b_m - \frac{\langle a_m,b_m\rangle}{\langle a_m,a_m \rangle}a_m \\
 d_m &=& \left(\sqrt{\frac{\langle a_m , a_m \rangle}{\langle c_m, c_m \rangle}}\right)c_m
\end{eqnarray*}
 Then $a_m + \sqrt{-1}d_m \in \PD{\LL,\ell^{\perp}}^{\mathrm{max}}$.
 By definition, $\lim_{m \to \infty} a_m + \sqrt{-1}d_m = t$ and we are done.
\end{proof}
Let $N$ be a positive integer and  $\LL_N$ represents the same object in 
Corollary \ref{density_of_period_2}. 
We define the second subset of $\PD{\LL, \ell^{\perp}}$ by
$$
 \PD{\LL, \ell^{\perp}}^{\circ} 
:= \{x \in \PD{\LL, \ell^{\perp}} | x^{\perp} \cap \ell^{\perp} \cap \LL_N = \emptyset\}.
$$
\begin{claim}\label{open_and_dense}
The subset $\PD{\LL,\ell^{\perp}}^{\circ}$ is open and dense.
\end{claim}
\begin{proof}
 For a very general point $x$ of $\in \PD{\LL, \ell^{\perp}}$, 
$x^{\perp} \cap \LL = \ZZ \ell$. Hence $x \in \PD{\LL , \ell^{\perp}}^{\circ}$
 and $\PD{\LL, \ell^{\perp}}^{\circ}$ is dense.
We have a natural identification
$$
 \PD{\LL,\ell^{\perp}} \cong \GR{2,\ell^{\perp}}
$$
where $\GR{2,\ell^{\perp}}$ is the set of positive $2$-planes 
in $\ell^{\perp}$. The correspondence is given
by
$$
 \PD{\LL , \ell^{\perp}} 
\ni t \mapsto \langle \mathrm{Re}(t), \mathrm{Im}(t) \rangle \in \GR{2,\ell^{\perp}},
$$
where $\langle \mathrm{Re}(t), \mathrm{Im}(t) \rangle$ is the $2$-plane spanned
by $\mathrm{Re}(t)$ and $\mathrm{Im}(t)$.
Under this identification, $\PD{\LL, \ell^{\perp}}^{\circ}$ corresponds to
$$
\GR{2,\ell^{\perp}}^{\circ} =
 \{ \sigma \in \GR{2,\ell^{\perp}} | \forall x \in \Lambda_N \cap \ell^{\perp}, 
\sigma \not\subset x^{\perp}\}
$$
By Corollary \ref{density_of_period_2}, the above set is open in $\GR{2,\ell^{\perp}}$
and we are done.
\end{proof}
Let $\pi :\XF \to \DEF{X}$ be a Kuranishi family of $X$. 
For a point $t$ of $\DEF{X}$, $\pi$ gives a natural marking $\varphi_t : H^2 (\XF_t , \ZZ) \to \LL$, where
$\XF_t$ is the fibre at $t$.
We consider
the subset of $\DEF{X}$ defined by
$$
 \DEF{X,L} := \{ t \in \DEF{X} | \PF (\XF_t , \varphi_t ) \in \PD{\LL, \ell^{\perp}} \}
$$
and the restriction family $\XF_L \to \DEF{X,L}$. By \cite[Corollaire 1]{MR785234},
$\XF_L$ carries a line bundle $\LF$ such that the restriction of $\LF$ to $X$ 
is isomorphic to $L$. Since the Betti number is greater than five, by 
\cite[Corollary 1.4]{2016arXiv160403927A},
there exists a positive 
integer $N$ such that $\varphi_t (\MBM{X_t}) \subset \LL_N$ for all $t \in \DEF{X}$. 
By Claim \ref{density_of_higher_picard_number}
and Claim \ref{open_and_dense}, there exists a point $t_0$ of $\DEF{X}$ such that
$\PF (\XF_{t_0},\varphi_{t_0}) \in \PD{\LL, \ell^{\perp}}^{\mathrm{max}} \cap \PD{\LL, \ell^{\perp}}^{\circ}$.
By Lemma \ref{fibre_property}, we have a marked irreducible symplectic manifold
$(X',\varphi')$ such that $X'$ carries a nef line bundle $L'$ with $\varphi' (c_1(L')) = \ell$
and $\PF (X',\varphi') = t_0$. Since $\RK{t_0^{\perp} \cap \LL} = \RK{\LL} - 2$, the Picard
number of $X'$ is equal to $\dim H^2 (X',\RR) - 2$. Since $t_0^{\perp}\cap \ell^{\perp} \cap \LL_N = \emptyset$,
$c_1 (L')^{\perp} \cap \MBM{X'} = \emptyset$. Thus we are done if we prove that
$(X,L)$ and $(X',L')$ are deformation equivalent in the sense of Definition \ref{Deformation_Equivalent}.
Let $\pi' : \XF' \to \DEF{X'}$ be a Kuranishi family of $X'$.
We consider the restriction family $\XF'_{L'} \to \DEF{X',L'}$ which is obtained by
the same manner of $\XF_L \to \DEF{X,L}$. Since
$t_0 \in \PF (\DEF{(X,L)}) \cap \PF (\DEF{X',L'})$,
$ \PF (\DEF{(X,L)}) \cap \PF (\DEF{X',L'})$
 is a non empty open subset of $\PD{\LL,\ell^{\perp}}$.
Hence there exists a point $t_1$ of $\PF(\DEF{X,L}) \cap \PF(\DEF{X',L'})$ such that 
$t_1^{\perp} \cap \LL = \ZZ \ell$.
Let  $\XF_{t_1}$ be the fibre of $\XF_L \to \DEF{X,L}$ at $t_1$ and 
$\XF'_{t_1}$ the fibre of $\XF'_{L'} \to \DEF{X',L'}$ at $t_1$.
We denote by $\varphi_{t_1}$ the induced marking on $H^2 (\XF_{t_1}, \ZZ)$
 and by $\varphi'_{t_1}$ the induced marking on $H^2 (\XF'_{t_1}, \ZZ)$.
Then $(\XF_{t_1}, \varphi_{t_1})$ and $(\XF'_{t_1}, \varphi'_{t_1})$
are isomorphic
by \cite[Theorem 2.2 (5)]{Markman_2011}. 
We denote by $\Phi_{t_1}$ an isomorphism between 
$( \XF_{t_1},\varphi_{t_1} )$ and $( \XF'_{t_1}, \varphi'_{t_1} )$.
Since $ \varphi_{t_1}^{-1}(\ell) =  c_1 (\LF_{t_1})$ and
$(\varphi'_{t_1})^{-1}(\ell) = c_1 (\LF'_{t_1})$,
$\Phi_{t_1}^*\LF'_{t_1} \cong \LF_{t_1}$.
Hence $(X,L)$ and $(X',L')$ are deformation equivalent
in the sense of Definition \ref{Deformation_Equivalent}.
\end{proof}
\begin{proof}[Proof of Theorem \ref{main_3}.] 
 By Proposition \ref{Nice_Period_Model}, we have a pair $(X',L')$ with deformation equivalent to $(X,L)$
 such that $L'$ is nef
 and $c_1 (L)^{\perp}\cap \MBM{X'} = \emptyset$. By Theorem \ref{main_2}, $\AUT{X',L'}$ is almost 
 abelian whose rank is equal to $\dim H^2 (X',\RR) - 2$.
\end{proof}

\bibliographystyle{plain}
\bibliography{matsushita}
\end{document}